\documentclass[jthenal,twoside,web]{ieeecolor}
\usepackage{generic}
\usepackage{cite}
\usepackage{amsmath,amssymb,amsfonts}
\usepackage{algorithmic}
\usepackage{graphicx,float}
\usepackage{textcomp}
\def\BibTeX{{\rm B\kern-.05em{\sc i\kern-.025em b}\kern-.08em
    T\kern-.1667em\lower.7ex\hbox{E}\kern-.125emX}}
\DeclareMathOperator*{\argmin}{arg\,min}
\newtheorem{theorem}{Theorem}
\newtheorem{corollary}[theorem] {Corollary}
\newtheorem{proposition}[theorem]{Proposition}
\newtheorem{lemma}[theorem]{Lemma}

\begin{document}
\title{Designing Robust Networks of Coupled Phase-Oscillators with Applications to the High Voltage Electric Grid}

\author{Shriya V. Nagpal $^{*}$, Gokul G. Nair,  Francesca Parise $^{\dagger}$, and C. Lindsay Anderson $^{\dagger}$, \IEEEmembership{Senior Member, IEEE}
\thanks{ $^{*}$ indicates the corresponding author and $^{\dagger}$ indicates equal contribution.}
\thanks{S. V. Nagpal and G. G. Nair are with the Center for Applied Mathematics, Cornell University, Ithaca, NY 14853 USA (e-mail: svn23@cornell.edu; gn234@cornell.edu ). }
\thanks{F. Parise is with the School of Electrical and Computer Engineering and the Center for Applied Mathematics, Cornell University, Ithaca, NY 14853 USA (e-mail: fp264@cornel.edu ).}
\thanks{C. L. Anderson is with the Department of Biological and Environmental Engineering and the Center for Applied Mathematics, Cornell University, Ithaca, NY 14853 USA (e-mail: cla28@cornell.edu).}}

\maketitle

\begin{abstract}
We propose a mathematical framework for designing robust networks of coupled phase-oscillators by leveraging a vulnerability measure proposed by Tyloo et. al that quantifies the impact of a small perturbation at an individual phase-oscillator's natural frequency blue to the system's global synchronized frequencies. Given a complex network topology with specific governing dynamics, the proposed framework finds an optimal allocation of edge weights that minimizes such vulnerability measure(s) at the node(s) for which we expect perturbations to occur by solving a tractable semi-definite programming problem. We specify the mathematical model to high voltage electric grids where each node corresponds to a voltage phase angle associated with a bus and edges correspond to transmission lines. Edge weights are determined by the susceptance values along the transmission lines. In this application, frequency synchronization is increasingly challenged by the integration of renewable energy, yet is imperative to the grid's health and functionality. Our framework helps to alleviate this challenge by optimizing the placement of renewable generation and the susceptance values along the transmission lines.
\end{abstract}


\begin{IEEEkeywords}
Coupled phase-oscillators, convex optimization, frequency synchronization, high voltage electric grid, robust network design, renewable energy, semidefinite programming
\end{IEEEkeywords}

\section{Introduction}
\label{sec:introduction}
\IEEEPARstart{C}{omplex } networks are frequently used to model coupled dynamical systems ranging from interacting molecules in chemical reactions~\cite{Filatrella} to high voltage electric grids~\cite{Tyloo2}. Be it man-made or natural, elements of a coupled dynamical system are represented by nodes in a complex network, and two nodes are adjacent to one another if the differential equations that govern those nodes, are dependent on one another~\cite{Tyloo2}. Two questions that are often investigated in complex networks are: 

\begin{enumerate}
    \item  What are the vulnerable nodes of the complex network?
    \item How can one use this knowledge to design robust complex networks?
\end{enumerate}

In this work, we seek to address the latter question for a complex network of coupled phase-oscillators. Specifically, we consider a weighted, connected, and undirected network, $G = (V,E)$, where $V$ is a set of $n$ nodes and $E$ is a set of $m$ edges. $B\in \mathbb{R}^{n \times n}$ is the weighted adjacency matrix specifying the edge weights of $G$; $B_{ij} = b_{ij}\geq{0}$ if $(i,j)\in{E}$ and $B_{ij}=0$ if $(i,j)\notin{E}$.\footnote{We allow for edge weights of the fixed network topology to be $0$ to facilitate an expansive proposed optimization framework.} Each node $i\in{V}$ in the network corresponds to an angle, $\theta_{i} \in[-\pi, \pi)$, that evolves according to the coupled dynamics

\begin{equation}\label{eqn:coupledoscillator1}
 \dot{\theta}_{i}=\omega_{i}-\sum_{{j\in\mathcal{N}(i)}} b_{i j} \sin \left(\theta_{i}-\theta_{j}\right), i=1, \ldots, n.
\end{equation}

\noindent
where $\mathcal{N}(i)$ is the set of all nodes, $j$, such that $(i,j)\in{E}$. The phase-oscillator's dynamics of node $i$ is determined by its natural frequency, $\omega_{i}$, and its coupling with other phase-oscillators determined by the network's edge weights $B$. Despite the apparent simplicity, the coupled phase-oscillator model and its variations have been utilized to describe and analyze a broad array of applications including circadian rhythms, flashing fireflies, and high voltage electric grids~\cite{Dorfler2}. In many of these phenomena, it is desirable for phase-oscillators to maintain global synchronized frequencies, i.e., $\dot{\theta}_{i} = \omega_0$ for all $i\in\{1,\dots,n\}$.

Following~\cite{Tyloo}, we measure vulnerability of a network by quantifying how much a small perturbation to a node's/phase-oscillator's natural frequency impacts the system's global synchronized frequencies. A small external perturbation at a node with high vulnerability has a larger influence on the global synchronized frequencies than nodes with a smaller vulnerability measure. Interestingly, in~\cite{Tyloo} the authors show that such a vulnerability measure may be written as a linear combination of generalized effective resistance measures. In the remainder of this section, we recap intuition into the derivation of this measure and then describe how to leverage this vulnerability measure for the purpose of designing robust systems.  

We apply the design framework to high voltage grids where each node $i$ corresponds to a voltage phase angle $\theta_{i} \in[-\pi, \pi)$, associated with a bus $i$, and evolves according to the coupled dynamics given in \eqref{eqn:coupledoscillator1},~\cite{Guo,Strogatz,Hu}. Here, the voltage phase-oscillators' ability to maintain synchronized frequencies is essential to the functionality of the grid. The current administration plans to have wind and solar energy comprise ninety percent of the United State's electricity profile by $2050$~\cite{Engel}, but this integration will likely result in small perturbations to the power injected into the system due to variability on renewable output \cite{Tyloo,Impram}, challenging the voltage phase-oscillators' capacity to maintain synchronized frequencies. This work seeks to address this tension by optimizing the placement of renewable generation and the susceptance values along the transmission lines to minimize the effect of disturbances on the voltage phase-oscillators' frequencies, in line with the proposed $2$ billion dollar government investment for clean energy infrastructure~\cite{Penrod}.

\textbf{Notation.} Let $\mathbb{R}$ denote the set of real numbers. We consider $b\in \mathbb{R}^{m}$ to be an $m$ length vector consisting of all $b_{ij}$ where $(i,j)\in{E}$ and $\mathbf{1}$ to be the all ones vector. The constraint $b\geq0$ is equivalent to enforcing that $b_{ij}\geq0$ where $(i,j)\in{E}$. $\mathbb{L}\in \mathbb{R}^{n \times n}$ is the network Laplacian matrix corresponding to $G = (V,E)$ where $\mathbb{L}_{i j}=-B_{i j}$ if $i \neq j$, and $\mathbb{L}_{i i} = \sum_{k} B_{i k}$, and $\mathbb{L}^{\dagger}\in \mathbb{R}^{n \times n}$ is the Moore–Penrose pseudoinverse of $\mathbb{L}$. Throughout the text we use $\mathbb{L}$ interchangeably with $\mathbb{L}(b)$ and $G$ interchangeably with $G(b)$ to remind the reader that properties of $\mathbb{L}$ and $G$ are dependent on the edge weights of the network. 
Similarly, suppose $\lambda_i$ is the $i^{th}$ eigenvalue of $\mathbb{L}$;  $\lambda_i(b)$ is used interchangeably with $\lambda_i$. Lastly, $Y\succ 0$ means that $Y$ is positive definite, while $Y\succeq{0}$ means that $Y$ is semi positive definite.
\color{black}



\subsection{Vulnerability Measure}
Let $\omega^{(0)}=\big[\omega^{(0)}_1,\dots,\omega^{(0)}_n\big]^{T}$ be a vector of natural frequencies. When natural frequencies, $\omega^{(0)}_i$ for all $i$, are not too large compared to their coupling parameters, stable solutions exist where phase-oscillators have global synchronized frequencies, i.e., $\dot{\theta}_{i} = \omega_0$ for all $i\in\{1,\dots,n\}$\footnote{This statement is made more precise later on in this write up.}~\cite{Tyloo3}. By working in a rotating reference frame, one may assume $\dot{\theta}_{i} = 0$ for all $i$, resulting in a stable fixed point, $\theta_{}^{(0)} = (\theta_{1}^{(0)},\dots,\theta_{n}^{(0)})$. Subjecting $\omega_{i}^{(0)}$ to a time dependent perturbation, $\omega_{i}(t)=\omega_{i}^{(0)}+\tilde{\omega}_{i}(t)$, results in phase angles becoming time-dependent, $\theta_{i}(t)=\theta_{i}^{(0)}+ \tilde{\theta}_{i}(t)$, and linearizing the dynamics around $\theta^{(0)}$ ultimately yields:

\begin{equation}\label{eqn:kuramoto2}
\dot{\tilde{\theta_i}} = \tilde{\omega}_i-\sum_{j\in\mathcal{N}(i)}b_{ij}\cos(\theta_i^{(0)}-\theta_j^{(0)})(\tilde{\theta_i}-\tilde{\theta_j}), i=1, \ldots, n. 
\end{equation}

Suppose $k$ is some node in the network. To determine the vulnerability measure of this node we set $\tilde{\omega}_{k}(t)$ to a time dependent, Ornstein-Uhlenbeck noise disturbance, and $\tilde\omega_s=0$ for all $s\neq k$. Let ${\tilde{\theta}}^{(k)}_{i}$ be a corresponding solution of \eqref{eqn:kuramoto2}.\footnote{We note that the  Ornstein-Uhlenbeck process can be considered as the continuous-time analogue of the discrete-time $AR(1)$ process.} We then define  

\begin{equation}\label{eqn:centrality_measure1}
\mathcal{M}_k(b) = \lim _{T \rightarrow \infty} T^{-1} \sum_{i} \int_{0}^{T} \overline{\left| \dot{\tilde{\theta}}^{(k)}_{i}(t)-\dot{\Delta}^{(k)}(t)\right|^{2}} \mathrm{~d} t 
\end{equation}

\noindent
where $\dot{\Delta}^{(k)}(t)=n^{-1} \Sigma_{j} \dot{\tilde{\theta}}^{(k)}_{j}(t)$, and the bar represents an average with respect to the random noise. Intuitively, this measure quantifies how much a specific perturbation at a node $k$ impacts the global angular-frequency synchronization; If the measure is small, the oscillators' frequencies remain synchronized, or at least close in value, throughout time when exposed to a small perturbation.\footnote{The authors in~\cite{Tyloo} do make an implicit assumption that the perturbation is small enough such that the dynamics remain within the basin of attraction.} While \eqref{eqn:centrality_measure1} depends on the solution of an ordinary differential equation, in~\cite{Tyloo} it is shown that $\mathcal{M}_k(b)$ can be expressed in terms of network properties only. Specifically, in~\cite{Tyloo}, the authors' derive an analytical expression for \eqref{eqn:centrality_measure1} and show that if the timescale of correlation in the noise is large in comparison to the dynamical system timescale, then 

\begin{equation}\label{eqn:centrality_measure}
    \mathcal{M}_k(b) = c_1 \Bigg(n^{-1}\sum_{j} \Omega_{j k}\left({\theta}^{(0)},b\right)-n^{-2}\sum_{i<j} \Omega_{i j}\left({\theta}^{(0)},b\right)\Bigg)
\end{equation}

\noindent
where $c_1$ is a fixed positive constant independent of the network, $b$ is a $m$ length vector that consists of all $b_{ij}$ where $(i,j)\in{E}$, and 

\[
    \Omega_{i j}\left({\theta}^{(0)},b\right)=\mathbb{L}_{i i}^{\dagger}\left({\theta}^{(0)}\right)+\mathbb{L}_{j j}^{\dagger}\left({\theta}^{(0)}\right)-2\mathbb{L}_{i j}^{\dagger}\left({\theta}^{(0)}\right)
\]

\noindent
is the effective resistance \cite{Tyloo} corresponding to the weighted network Laplacian matrix evaluated at a steady state, whose entries are given by
\[
    \mathbb{L}_{i j}\left({\theta}^{(0)}\right)=-B_{i j}\cos \left(\theta_{i}^{(0)}-\theta_{j}^{(0)}\right)
\]

\noindent
if $i \neq j$ and 
\[
    \mathbb{L}_{i i}\left({\theta}^{(0)}\right)= \sum_{k} {B_{i k}} \cos \left(\theta_{i}^{(0)}-\theta_{k}^{(0)}\right). 
\]

The same vulnerability measure may be derived by exposing oscillator $k$'s natural frequency to a temporary box noise perturbation, expanding the types of perturbations for which the vulnerability measure accounts for \cite{Tyloo3}.To further highlight the broad applicability of this measure, it is shown in \cite{Tyloo} that the vulnerability measure may also be considered for a network of coupled phase-oscillators with the following second-order dynamics

\begin{equation*}
 {m_i}\ddot{\theta}_{i} + {d_i}\dot{\theta}_{i}=\omega_{i}-\sum_{j\in\mathcal{N}(i)} b_{i j} \sin \left(\theta_{i}-\theta_{j}\right), 
\end{equation*}
\noindent
where $m_i = m_0$ and $d_i = d_0$ for all $i=1, \ldots, n$, and is additionally numerically justified for independently varying $m_i$ and $d_i$.
\color{black}

\subsection{Existing Literature and Contributions}

Various topological vulnerability measures have been discussed in the literature that quantify the ability of \textit{identical} oscillators (i.e. $\omega_{i}^{(0)}=\bar{\omega}$ for all $i$) to maintain synchronized frequencies in the presence of small perturbations~\cite{Fazlyab,Donetti,Tyloo3}. One such topological measure that has received considerable attention is the eigenratio of the network Laplacian, $Q = \lambda_{n} / \lambda_{2}$, where ${0}=\lambda_{1} < \lambda_{2}\leq \dots \leq \lambda_{n}$ are the eigenvalues of $\mathbb{L}$. Using the master stability framework proposed in Pecora et. al~\cite{Pecora}, it was shown that the interval in which the synchronized state is stable is larger for smaller $Q$~\cite{Donetti}. Following this work, many papers have leveraged this measure for the purpose of exploring and designing robust systems governed by \eqref{eqn:coupledoscillator1}~\cite{Kempton,Donetti,Hong,Pecora}. However, many of these design frameworks are rendered incompatible with existing applications because a number of these applications involve a network of \textit{nonidentical} coupled oscillators, like high voltage electric grids.

To remedy this drawback, this work considers \textit{nonidentical} coupled oscillator networks and seeks to \textit{design} robust oscillator networks capable of maintaining global synchronized frequencies in the presence of noise by considering the vulnerability measure \eqref{eqn:centrality_measure}. To this end, the main contribution of this paper is threefold. First, under a small phase angle difference assumption, the vulnerability measure \eqref{eqn:centrality_measure} can be written as a linear combination of effective resistance measures. This is particularly useful because a vast literature  exists on the effective resistance measure~\cite{Ghosh,Gharan,VanMieghem,ellens2011effective}, which we exploit to propose a mathematical model for designing robust networks of nonidentical coupled phase-oscillators that can be solved optimally and efficiently.  Second, in proposing this model, this work synthesizes two well studied bodies of work (work on networks of coupled phase-oscillators and work on the effective resistance measure) setting the stage for further interdisciplinary research of this type. Third, motivated by the need to integrate renewable energy into the grid given the growing threat of climate change, this work specifies the proposed mathematical model to the high voltage electric grid to facilitate the design of a system that is robust to the integration of renewable energy.

The rest of the paper is divided as follows. In Section~\ref{Problem Formulation}, the main mathematical problem is posed: Given a connected, complex network topology with governing dynamics described by~\eqref{eqn:coupledoscillator1}, the suggested framework seeks to find an optimal allocation of edge weights to minimize the vulnerability measures corresponding to a subset of nodes for which small perturbations are expected to occur. Section~\ref{Problem Formulation} provides sufficient conditions for the main problem to be convex and considers the worst case vulnerability measure for the purpose of designing a robust network of coupled phase-oscillators. Under this specification, the convex optimization problem is reformulated as a semidefinite programming problem (SDP) for the purpose of tractability. Section~\ref{Interpretations} provides further intuition into the optimization problem via the analysis of the vulnerability measure from both a graph theoretic and analytical perspective. Lastly, in Section~\ref{High Voltage Electric Grid Application}, the optimization framework is applied to the high voltage electric grid demonstrating not only the theoretical contribution of this work, but its practical utility as well.


\color{black}
\section{Problem Formulation}\label{Problem Formulation} 
Given a connected complex network topology with governing dynamics described by \eqref{eqn:coupledoscillator1}, this work's objective is to find an optimal allocation of edge weights to minimize the vulnerability measure at a node, or a function of vulnerability measures corresponding to a subset of nodes, for which we expect small perturbations to occur, subject to three constraints;The edge weights are non-negative, the edge weights sum to $1$,\footnote{This assumption is made for simplicity, but in reality, the edge weights may sum to any positive constant.} and the network remains connected. Specifically, the nodes for which we expect perturbations to occur is described by the subset $V'\subseteq{V}$, and in the application of interest, corresponds to buses where renewable energy is introduced. Edge weights correspond to susceptance values along the transmission lines and this work investigates how the distribution of susceptance values along the transmission lines of a high voltage electric grid topology facilitates robustness of that network. 

To start, we work under the assumption that the difference between phase angles of a steady state, $|\theta^{(0)}_i-\theta^{(0)}_j|$, is small for all $(i,j)\in{E}$ implying that $\cos(\theta^{(0)}_i-\theta^{(0)}_j)\approx{1}$.\footnote{Note that in the power grid, fixed points of the system are sought such that $|\theta^{(0)}_i-\theta^{(0)}_j|$ are small for all $(i,j)\in{E}$.} This means that for each $k\in V'$, $\frac{\mathcal{M}_k(b)}{c_1}\approx \hat{\mathcal{M}}_k(b) := n^{-1}\sum_{j} \Omega_{j k}(b)-n^{-2}\sum_{i<j} \Omega_{i j}(b)$ where $\Omega_{i j}(b)=\mathbb{L}_{i i}^{\dagger}+\mathbb{L}_{j j}^{\dagger}-2\mathbb{L}_{i j}^{\dagger}$ is the effective resistance corresponding to the network Laplacian, $\mathbb{L}_{i j}=-B_{i j}$ if $i \neq j$, and $\mathbb{L}_{i i} = \sum_{k} B_{i k}$. Suppose $V' = \{k^{1},\dots,k^{l}\}$ where $l\leq{n}$. We wish to find $b^{*}\in \mathbb{R}^{m}$ to minimize a function of the vulnerability measures of nodes in $V'$, $\mathcal{{F}}: \mathbb{R}^{l}\rightarrow \mathbb{R}$, that is

\begin{align}\label{eqn:main_problem}
b^{*} = \argmin_{b\in\mathcal{X}} \mathcal{{F}} \bigg( \hat{\mathcal{M}}_{k^1}(b),\dots,\hat{\mathcal{M}}_{k^l}(b)\bigg)
\end{align}

where,
\begin{align*}
    \mathcal{X}=\{b\in\mathbb{R}^m:b\geq 0,\ b^T\mathbf{1}=1, G(b)\text{ is connected}\}.
\end{align*}

In the following section, sufficient conditions on $\mathcal{{F}}$ are provided in order to describe when \eqref{eqn:main_problem} is a convex optimization problem. This work focuses on the case where
$$\mathcal{F}\bigg(\hat{\mathcal{M}}_{k^1}(b),\dots,\hat{\mathcal{M}}_{k^l}(b)\bigg) = \max_{k\in{V'}}\hat{\mathcal{M}}_k(b)$$
\noindent to produce edge weights that optimally minimize the worst case vulnerability measure of nodes in $V'$. With this specification of $\mathcal{F}$, \eqref{eqn:main_problem} is reformulated as an SDP problem for the purpose of employing efficient solvers, simultaneously ensuring that the resultant edge weight assignment is such that the oscillators' frequencies synchronize and have small phase angle differences.


\section{Optimization Framework}

\subsection{Convex Optimization Problem}\label{Convex Optimization Problem}

Proposition~\ref{prop:convexity_of_X} shows that the set $b\in\mathbb{R}^m$ such that $G(b)$ is connected is a convex set. From this proposition, it immediately follows that $\mathcal{X}$ is a convex set since the intersection of convex sets are convex \cite{Boyd}.

\begin{proposition}\label{prop:convexity_of_X}
$\{b\in\mathbb{R}^m: G(b)\text{ is connected}\}$ is a convex set.
\end{proposition} 

\medskip

\begin{proof}
Suppose $\sigma\in[0,1]$ and let $b_1,b_2\in\mathbb{R}^m$ be such that $G(b_1)$ and $G(b_2)$ are both connected. When $\sigma=0$, ${G(\sigma{b_1}+(1-\sigma)b_2)} = G(b_2)$, and $G(b_2)$ is connected by assumption. 
Similarly, when $\sigma=1$, ${G(\sigma{b_1}+(1-\sigma)b_2)} = G(b_1)$, and $G(b_1)$ is connected by assumption. Lastly, when $\sigma\in(0,1)$, ${G(\sigma{b_1}+(1-\sigma)b_2)}$ consists of all of the edges in $G(b_1)$ and $G(b_2)$, and therefore ${G(\sigma{b_1}+(1-\sigma)b_2)}$ is connected since both $G(b_1)$ and $G(b_2)$ are connected.
\end{proof}

\medskip

\noindent {Theorem~\ref{thm:convexity-of-M} shows that $\hat{\mathcal{M}}_k(b) $ is convex with respect to the edge weights, $b\in{\mathcal{X}}$. From this theorem, sufficient conditions on $\mathcal{{F}}$ are provided in Corollary~\ref{cor:convex-composition} for when \eqref{eqn:main_problem} is a convex optimization problem, and for which our choice of $\mathcal{{F}}$ satisfies.

\medskip

\begin{theorem}\label{thm:convexity-of-M}
Suppose $G=(V,E)$ is a simple, connected network with $|V| = n$  and $|E| = m$. Let $b_{ij}\geq0$ be the edge weight that corresponds to edge $(i,j)\in{E}$ and suppose $b$ is a $m$ length vector that consists of all $b_{ij}$. For each node $k\in{V}$, the measure $ \hat{\mathcal{M}}_k(b) $ is convex with respect to the edge weights $b\in{\mathcal{X}}$.
\end{theorem}

\begin{proof} Since $\mathbb{L}$ is a real symmetric matrix, it can be written as $\mathbb{L} = {V}\Lambda{V^{T}}$ where $V \in \mathbb{R}^{n \times n}$, $\Lambda \in \mathbb{R}^{n \times n}$ are such that:

\begin{itemize}
    \item $V^{T}V = VV^{T} = I$ 
    \vskip2mm
    \item $\Lambda=\left[\begin{array}{cccc}\lambda_{1} & 0 & \ldots & 0 \\ 0 & \lambda_{2} & \ldots & 0 \\ \vdots & 0 & \ddots & 0 \\ 0 & \ldots & 0 & \lambda_{n}\end{array}\right]$ 
\end{itemize}

\noindent
where ${0}=\lambda_{1} < \lambda_{2}\leq \lambda_{3} \dots \leq \lambda_{n}, \lambda_{i}\in \text{spec}(\mathbb{L})$. Moreover, the Moore–Penrose pseudoinverse of $\mathbb{L}$ may be written as $\mathbb{L}^{\dagger}=\left(\mathbb{L}+\mathbf{1 1}^{T} / n\right)^{-1}-\mathbf{1 1}^{T} / n$~\cite{Ghosh}, or $\mathbb{L}^{\dagger} = {V}\Lambda^{\dagger}{V^{T}}$, for

\begin{center}
   $\Lambda^{\dagger} =\left[\begin{array}{cccc}0 & 0 & \ldots & 0 \\ 0 & \frac{1}{\lambda_{2}} & \ldots & 0 \\ \vdots & 0 & \ddots & 0 \\ 0 & \ldots & 0 & \frac{1}{\lambda_{n}}\end{array}\right]$ .
\end{center}

\noindent
Now, for fixed $k$, consider the measure $\hat{\mathcal{M}}_k(b) =  n^{-1}\sum_{j} \Omega_{j k}(b)-n^{-2}\sum_{i<j} \Omega_{i j}(b)$. Recall that $e_k$ is a standard basis vector of length $n$, and let $v_\alpha$ denote the eigenvector associated with $\alpha^{th}$ eigenvalue of $\mathbb{L}$, i.e, the $\alpha^{th}$ column of $V$. Then,

\begin{align*}
    \hat{\mathcal{M}}_k(b) &=  n^{-1}\sum_{j} \Omega_{j k}(b)-n^{-2}\sum_{i<j} \Omega_{i j}(b)\\
    &\stackrel{*}{=} \Bigg(\sum_{\alpha \geq 2} \frac{v_{\alpha{k}}^{2}}{\lambda_{\alpha}}+n^{-2}\sum_{i<j} \Omega_{i j}(b)\Bigg) - n^{-2}\sum_{i<j} \Omega_{i j}(b)  \\
    &= \sum_{\alpha \geq 2} \frac{v_{\alpha{k}}^{2}}{\lambda_{\alpha}} = \mathbb{L}^{\dagger}_{kk} = {e_k}^{T}\mathbb{L}^{\dagger}{e_k} \\
    &= {e_k}^{T}\big(\left(\mathbb{L}+\mathbf{1 1}^{T} / n\right)^{-1}-\mathbf{1 1}^{T} / n\big){e_k} \\
    &= {e_k}^{T}\left(\mathbb{L} +\mathbf{1 1}^{T} / n\right)^{-1}{e_k} - {e_k^{T}}\big(\mathbf{1 1}^{T} / n\big){e_k}\\
    &={e_k}^{T}\left(\mathbb{L}+\mathbf{1 1}^{T} / n\right)^{-1}{e_k} - \frac{1}{n}.
\end{align*}

\noindent
For an explanation of the first equality (denoted with $*$) we refer the reader to \cite{VanMieghem}. From~\cite{Ghosh}, we observe that $f(Y)=c^{T} Y^{-1} c$, where $Y=Y^{T} \in \mathbb{R}^{n \times n}$ and $c \in \mathbb{R}^{n}$, is a convex function of $Y$ for $Y \succ 0$. Consequently, since $\left(\mathbb{L} +\mathbf{1 1}^{T} / n\right)^{-1} \succ 0$~\cite{Ghosh} and $\mathbb{L}+\mathbf{1 1}^{T} / n$ is an affine function of the edge weights of the graph $G$, ${e_k}^{T}\left(\mathbb{L}+\mathbf{1 1}^{T} / n\right)^{-1}{e_k}$ is a convex function of the edge weights of the graph. This means that ${e_k}^{T}\left(\mathbb{L} +\mathbf{1 1}^{T} / n\right)^{-1}{e_k} - \frac{1}{n}$ is a convex function of the edge weights of the graph since $\frac{1}{n}$ is simply a constant. 
\end{proof}
\medskip

\begin{corollary}\label{cor:convex-composition} 
If $\mathcal{{F}}: \mathbb{R}^{|V'|}\rightarrow \mathbb{R}$ is convex and nondecreasing, then $\mathcal{{F}} \bigg( \hat{\mathcal{M}}_{k^1}(b),\dots,\hat{\mathcal{M}}_{k^l}(b)\bigg)$ is convex with respect to $b$.
\end{corollary}

\noindent
Corollary~\ref{cor:convex-composition} is proven in \cite{Boyd} and implies that $\max_{k\in{V'}}\hat{\mathcal{M}}_k(b)$ is convex with respect to the edge weights of $b$~\cite{Ghosh}.

\subsection{SDP Formulation}\label{SDP Formulation}
We now show how to reformulate problem \eqref{eqn:main_problem} (specified to this work's choice of $\mathcal{{F}}$) into an SDP problem so that efficient solvers may be employed. To start, a property of the objective function (Proposition~\ref{prop:lower-bound-on-M}) is proven for the purpose of incorporating the connectivity constraint in a way that is compatible with SDP formulations. In particular, the fact that $G(b)$ is connected if and only if the second smallest eigenvalue of the corresponding network Laplacian, $\lambda_2(b)$, is positive is leveraged. The constraint $\lambda_2(b)>0$ cannot directly be ensured in the SDP. We instead use $\lambda_2(b)>\epsilon$ and show that this is, without loss of generality, equivalent for sufficiently small $\epsilon$.

\begin{proposition}\label{prop:lower-bound-on-M}
For any $b\in{\mathcal{X}}$ and any ${V'}\subset{V}$, 
$$\max_{k\in{V'}}\hat{\mathcal{M}}_k(b)\geq \frac{1}{\lambda_2(b)} \left(1-\frac{1}{n}\right)^{2}.$$
\end{proposition} 

\medskip

\noindent

For a proof of Proposition~\ref{prop:lower-bound-on-M}, please see Section~\ref{sec:proof-of-prop} of the Appendix.

\medskip

\noindent
By Proposition~\ref{prop:lower-bound-on-M}, there exists an $\epsilon>0$ such that for all $b$ where $0<\lambda_2(b)<\epsilon$, $\max_{k\in{V'}}\hat{\mathcal{M}}_k(b)>\max_{k\in{V'}}\hat{\mathcal{M}}_k(b^{*})$. Leveraging this $\epsilon$ and results from Theorem~\ref{thm:convexity-of-M}, \eqref{eqn:main_problem} (specified to this work's choice of $\mathcal{{F}}$) may be written as:

\begin{align}\label{eqn:main_problem_specified}
\argmin_{b\in\mathcal{X}} \max_{k\in{V'}}  {e_k}^{T}\left(\mathbb{L}+\mathbf{1 1}^{T} / n\right)^{-1}{e_k}
\end{align}

where,
\begin{align*}
    \mathcal{X}=\{b\in\mathbb{R}^m:b\geq 0,\ b^T\mathbf{1}=1, \mathcal{E}\succeq{0}\}.
\end{align*}

Note that the constraint $\mathcal{E} = \mathbb{L}+\mathbf{1 1}^{T} / n - \epsilon{I}\succeq{0}$ guarantees that all of the eigenvalues of $\mathbb{L}+\mathbf{1 1}^{T} / n$ are greater than or equal to $\epsilon$, and in particular, $\lambda_2(b)\geq\epsilon>0$. Next, set a slack variable $t$ such that

$$
    {e_k}^{T}\left(\mathbb{L}+\mathbf{1 1}^{T} / n\right)^{-1}{e_k}\leq t \text{ for all } k\in{V'}.
$$ 
By Schur's complement, ${e_k}^{T}\left(\mathbb{L}+\mathbf{1 1}^{T} / n\right)^{-1}{e_k}\leq t$ for all $k\in{V'}$ if and only if

$$
S_k :=
\left[\begin{array}{ll}
\mathbb{L}+\mathbf{1 1}^{T}/n & e_{k} \\
e^{T}_{k} & t
\end{array}\right]{\succeq{0}} \text{ for all } k.
$$

\noindent
Hence, problem \eqref{eqn:main_problem_specified}, may be written as 

\begin{align*}
\argmin_{b\in\mathcal{X}, S_k\succeq{0}\,\forall k\in V'} t
\end{align*}

Construct the block diagonal matrix $Z\in \mathbb{R}^{d\times{d}}$ where $d = |V^{'}|(n+1)+m+n$ as:

\begin{center}
$Z=\left[\begin{array}{ccccccccc}
S_1&&&&&&&&  \\ 
&S_2 &&&&&  \\
&&\ddots &&&&& \\ 
&&&S_{|V'|} &&&  \\
&&&&b_1 && \\
&&&&&b_2 &   \\
&&&&&&\ddots\\
&&&&&&&b_m\\
&&&&&&&&\mathcal{E}\\
\end{array}\right]$.
\end{center}

Define $W\in \mathbb{R}^{d\times{d}}$ such that $W_{ij} = 1$ if and only if $i = j = n+1$, and $W_{ij} = 0$ otherwise. Define $A\in \mathbb{R}^{d\times{d}}$ such that $A_{ij} = 1$ if and only if

$$i\in\{(n+1)|V'|+1,\dots,(n+1)|V'|+m\}$$

\noindent
and $i=j$, and $A_{ij} = 0$ otherwise. Then, coalescing everything together, results in the following SDP formulation of problem \eqref{eqn:main_problem_specified} :

\begin{align}\label{eqn:SDP_formulation}
 \argmin_{b} &\operatorname{Tr}(WZ)\\
 \text{s.t. }&\operatorname{Tr}(AZ)=1,\nonumber\\
 &Z\succeq{0}\nonumber.
\end{align}
\noindent
To solve \eqref{eqn:SDP_formulation} for $b^{*}$, we utilize CVXPY, a Python-embedded modeling language for convex optimization problems~\cite{CVXPY}. 

\subsection{Ensuring Synchronicity}\label{Ensuring Synchronicity}
In the derivation of the vulnerability measure, we assumed that the oscillators have synchronized frequencies and small phase angle differences. This assumption is not always true and is dependent on the relationship between the oscillators' natural frequencies and the edge weights of the complex network. The work~\cite{Bullo} delineates a condition relating complex networks oscillators' natural frequencies and edge weights that ensures the existence of a stable synchronized solution with phase angle differences less than a given small angle parameter, $\gamma$. Specifically, the authors show that if the second smallest eigenvalue, $\lambda_2(b)$ of $\mathbb{L}$ satisfies

    \begin{equation}\label{eqn:constraint3}
        \lambda_2(b)\geq{\|\omega^{(0)}\|_{{E}, \infty} \cdot \sin (\gamma)},
    \end{equation}

\noindent
where $\|x\|_{{E}, \infty}=\max _{\{i, j\} \in {E}}\left|x_{i}-x_{j}\right|$, then $\lvert\theta^{(0)}_i-\theta^{(0)}_j\rvert\leq\gamma$ for a wide class of networks. Thus, to ensure that a synchronized stable solution with angle difference less than a chosen, small angle parameter $\gamma$ exists,~\eqref{eqn:constraint3} is incorporated as a constraint into the optimization framework by further tightening the already existing constraint, $\mathcal{E}\succeq{0}$. That is, set $\epsilon = \|\omega^{(0)}\|_{{E}, \infty} \cdot \sin (\gamma)$, and note that $\mathcal{E}\succeq{0}$ is a sufficient condition for~\eqref{eqn:constraint3}.

\section{Interpretations}\label{Interpretations}

To start to shed light on the edge weight assignments that result from the optimization framework, \eqref{eqn:main_problem_specified}, specified to this work's choice of $\mathcal{F}$, this section investigates how to optimally assign edge weights to minimize vulnerability of a specific node, $k$, 
\begin{align}\label{eqn:main_problem_one_node}
\argmin_{b\in\mathcal{X}} \hat{\mathcal{M}}_{k}(b).
\end{align}

\noindent Section~\ref{Graph Theoretic Analysis} reformulates the vulnerability measure in terms of expected commute time, a graph theoretic measure. This reformulation facilitates intuition into how to optimally assign edge weights to solve~\eqref{eqn:main_problem_one_node}
and how this edge weight assignment permits a system of coupled phase-oscillators to be robust to small perturbations at node $k$. In Section~\ref{Analytical Solutions for Canonical Graphs}, a sufficient condition for the optimality of $b\in{\mathcal{X}}$ with respect to \eqref{eqn:main_problem_one_node} is derived and used to justify a prescribed edge weight assignment as optimal for a complete graph and any tree network of size $n$. 

\color{black}

\subsection{Graph Theoretic Analysis}\label{Graph Theoretic Analysis}
Let $D$ be the weighted degree matrix of $G$ where $D_{ij}=\sum_{j=1}^{n} B_{i j}$ for $i = j$ and $0$ otherwise. Define a discrete-time transition probability matrix $P = D^{-1}B$. Such a transition probability matrix defines a random walk on the weighted graph $G$ in which a random walker at node $i$ has probability $P_{i j}$ of visiting node $j$ in the next time step. The \textit{expected hitting time}, $H_{ij}$, is the expected number of steps such a random walker takes to reach node $j$ for the first time, starting from node $i$. The \textit{expected commute time} is the expected number of steps a random walker takes to reach node $j$, starting at node $i$, and then return to node $i$; $C_{ij} = H_{ij} + H_{ji}$ \cite{VanMieghem}. The following well-known result relates expected commute times and effective resistance \cite{Ghosh}:

$$
 C_{i j}=2\left(\mathbf{1}^{T} b\right) \Omega_{i j}(b) = 2 \Omega_{i j}(b)
$$

\noindent
since  $\mathbf{1}^{T} b = 1$. By multiplying both sides of 
$$\hat{\mathcal{M}}_k(b) =  n^{-1}\sum_{j} \Omega_{j k}(b)-n^{-2}\sum_{i<j} \Omega_{i j}(b)$$

\noindent
by $n^{2}$, we have 

$$n^2\hat{\mathcal{M}}_k(b) = \frac{1}{2}(n-1)\sum_{j} {C_{j k}}-\frac{1}{2}\sum_{\substack{i<j \\ i,j \neq {k}}} {C_{i j}}.$$

\noindent
Thus, $2n^{2}\hat{\mathcal{M}}_k(b) = (n-1)\sum_{j} {C_{j k}}-\sum_{\substack{i<j \\ i,j \neq {k}}} {C_{i j}}$ and so minimizing the vulnerability measure $\hat{\mathcal{M}}_k(b)$ at a node $k$ according to the optimization framework amounts to an edge weight assignment that seeks to (1) minimize the expected commute time from node $k$ to any other node in the system, and/or (2) maximize the expected commute time between any two nodes in the system that are not $k$. Note that (1) and (2) have different influences on an optimal edge weight assignment based on the \textit{network structure considered} and \textit{the node for which the vulnerability measure is being minimized at}. Figure~\ref{fig1} explores the influences of (1) and (2) on an optimal edge weight assignment for a few canonical graphs.}

\begin{figure}[H]
\centerline{\includegraphics[width=\columnwidth]{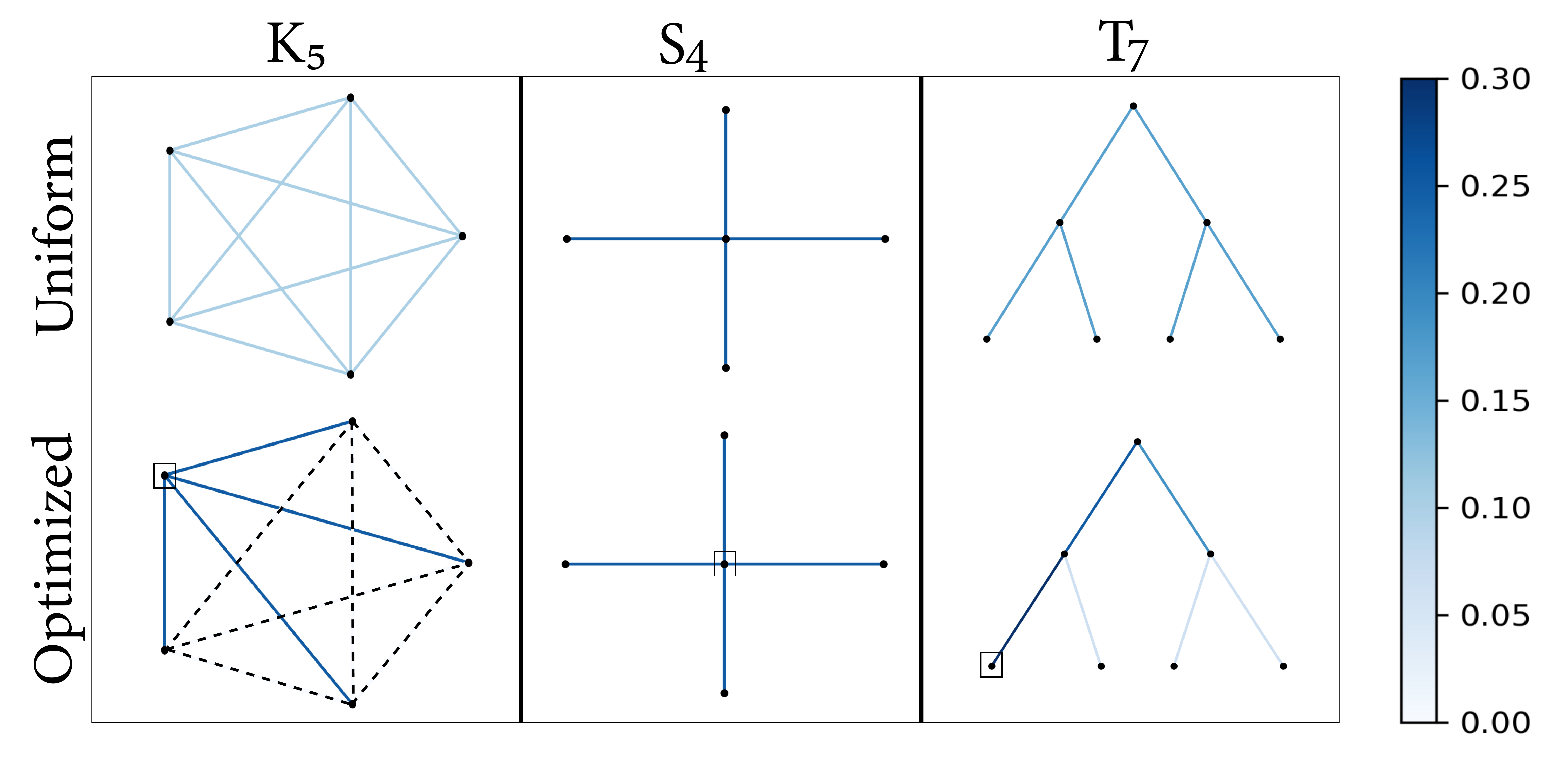}}
\caption{We consider a complete graph ($K_5$), a star graph ($S_4$), and a tree ($T_7$). Each graph in the top row has the same edge weight for each edge such that all edges sum to one, and in the bottom row, we apply this optimization framework to exactly one node, node $k$, in the graph indicated by the black square. The vulnerability measure at each node, $k$, after the optimization framework either decreases or stays the same. If edges are dotted, this means that the edge weight is $0$ at that edge.}
\label{fig1}
\end{figure}


In order to achieve the optimal vulnerability measure for node $k$ in the scenario ($K_5$), edge weights are assigned such that the sum of the expected commute time from node $k$ to any other node in the system decreases and the sum of the expected commute time between any two nodes in the system that are not $k$ increases. In ($S_4$), the edge weights, and therefore commute times, do not change. In ($T_7$), the sum of expected commute time from node $k$ to any other node in the system decreases, and as a result, the sum of the expected commute time between any two nodes in the system that are not $k$ decreases in order to achieve an optimal vulnerability measure at the chosen node $k$. 

To give insight on how assigning edge weights such that this graph theoretic/commute time interpretation is satisfied permits a system of oscillators' whose synchronized frequencies are robust to perturbations via simulation we consider a complete graph on five nodes with  uniform edge weights and synchronized frequencies and expose a specific oscillator's/node's natural frequency to a temporary box noise perturbation. Then, we consider the same complete graph on five nodes with edge weights obtained from the optimization framework and synchronized frequencies and expose the same node's natural frequency to the same temporary box noise perturbation. For both cases, we plot in Figure~\ref{fig2} the oscillators' frequencies over time for the purpose of comparison.

\begin{figure}
\centerline{\includegraphics[width=\columnwidth]{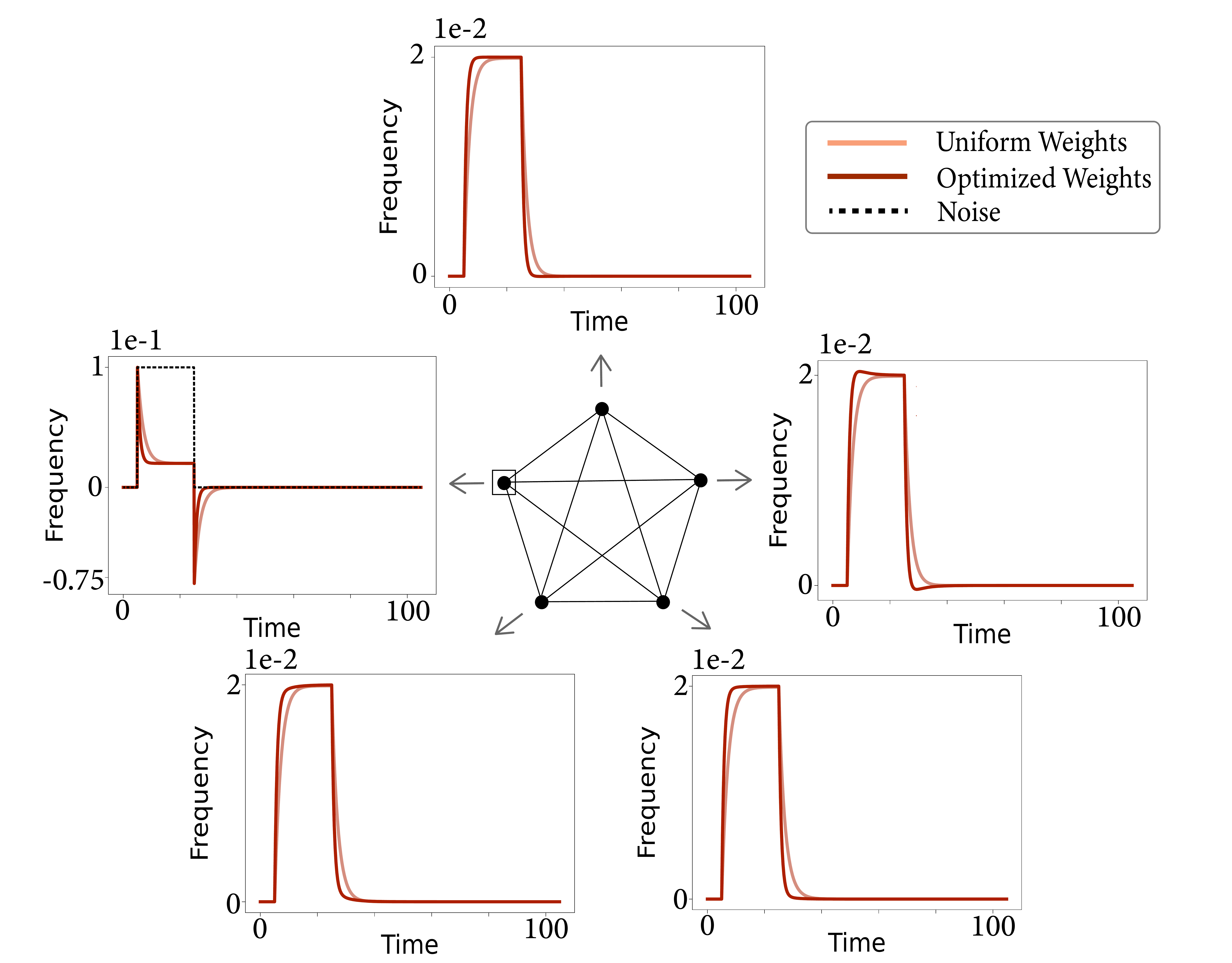}}
\caption{The node that is squared on the complete graph on five nodes in the center of this image is the node with perturbed natural frequency for both the uniform edge weight case ($b_0$) and optimized edge weight case ($b^{*}$). Each node on the graph has an associated arrow which points to a plot where the oscillators' frequencies over time for both cases, uniform and optimized edge weights, are considered. For each of these plots, we consider time (seconds) on the $x$-axis and frequency (in a co-rotating frame) on the $y$-axis. Notably, $\mathcal{M}_k(b_0)$ is about $2.4$ times larger than $\mathcal{M}_k(b^{*})$ when considering node $k$ indicated in Figure~\ref{fig1}}
\label{fig2}
\end{figure}

It is interesting to note that for all five nodes, the oscillators' frequencies behavior in the uniform edge weight case and optimized edge weights case are qualitatively similar except for one caveat; The oscillators frequencies' behavior in the uniform edge weights case seems to slightly lag the oscillators frequencies' behavior in the optimized edge weights case. We observe this qualitative behavior when performing the same type of simulation on $T_7$ and a figure elucidating this can be seen in Section~\ref{Graph Theoretic Interpretation Continued} of the Appendix.

This nicely connects the graph theoretic interpretation to the vulnerability measure \eqref{eqn:centrality_measure1} definition, $\mathcal{M}_k(b)$, which is how this work defines robustness in a network of coupled phase-oscillators.\color{black}\footnote{Here, recall that vulnerability measure \eqref{eqn:centrality_measure1} is derived in \cite{Tyloo3} for box noise perturbations.} In assigning edge weights such that the commute time interpretation is satisfied we are permitting perturbations introduced at node $k$ to have more of an influence and, therefore, propagate to other nodes at a quicker speed. This allows for oscillators' frequencies to synchronize at a quicker rate, or at least remain close in value, throughout time when a node is exposed to a small perturbation which is in line with the definition of the vulnerability measure \eqref{eqn:centrality_measure1}. \color{black}We observe this same qualitative behavior when node $k$'s natural frequency is exposed to an Ornstein-Uhlenbeck noise disturbance, however, the box noise perturbation lends itself to a sharper illustration.


\subsection{Analytical Formulations for Canonical Graphs}\label{Analytical Solutions for Canonical Graphs}

\begin{lemma} Let $b^{*}\in{\mathcal{X}}$. If for all $l\in\{1,\dots{m}\}$
\begin{equation}\label{eqn:optimality-condition}
    \frac{\partial\hat{\mathcal{M}}_k(b^{*})}{\partial b^{}_{l}} + \hat{\mathcal{M}_k}^{}(b^{*})\geq 0\tag{C1}
\end{equation}
then $b^{*}$ is an optimal solution to \eqref{eqn:main_problem_one_node}.
\end{lemma}

\begin{proof} 
We start by proving in Lemma~\ref{gradient_of_measure_property} in Section~\ref{Analytical Solutions for Canonical Continued} of the Appendix that

$$\left(\nabla \hat{\mathcal{M}}_k^{}(b)\right)^{T} b=-\hat{\mathcal{M}}_k^{}(b)$$ 
\noindent
for all $b\in\mathcal{X}$. Leveraging this property and (\ref{eqn:optimality-condition}), for all $l\in\{1,\dots{m}\}$,

\begin{align*}
\frac{\partial\hat{\mathcal{M}}_k(b^{*})}{\partial b^{}_{l}} &+ \hat{\mathcal{M}}_k^{}(b^{*})\geq 0\\
&\implies \frac{\partial \hat{\mathcal{M}}_k\big(b^{*}\big)}{\partial b^{}_{l}} - (-\hat{\mathcal{M}}_k(b^{*}))\geq 0\\
&\implies \left(\nabla \hat{\mathcal{M}}_k(b^{*})\right)^{T} e_{l} -\left(\nabla \hat{\mathcal{M}}_k(b^{*})\right)^{T} b^{*} \geq 0\\
&\implies \left(\nabla \hat{\mathcal{M}}_k(b^{*})\right)^{T} (e_{l}-b^{*})\geq 0
\end{align*}

\noindent
where $e_l$ is a standard basis vector of length $m$. Suppose $\bar{b}\in\mathcal{X}$ and note that we may write

$$\bar{b} = \sum_{i=1}^{m}{c_i}{e_i}$$ 

\noindent where $c_i\geq{0}$ and $\sum_{i=1}^{m}{c_i}=1$, since $\sum_{i=1}^{m}{\bar{b}_i}=1$. Since $\left(\nabla \hat{\mathcal{M}}_k(b^{*})\right)^{T} (e_{l}-b^{*})\geq 0$,

\begin{align*}
&\sum_{i=1}^{m}c_i\left(\nabla \hat{\mathcal{M}}_k(b^{*})\right)^{T} (e_{i}-b^{*})\geq 0\\
&\implies \left(\nabla \hat{\mathcal{M}}_k(b^{*})\right)^{T}\sum_{i=1}^{m}c_i{e_{i}}-\left(\nabla \hat{\mathcal{M}}_k(b^{*})\right)^{T}\sum_{i=1}^{m}c_i{b^{*}}\geq 0\\
&\implies \left(\nabla \hat{\mathcal{M}}_k(b^{*})\right)^{T}\bar{b}-\left(\nabla \hat{\mathcal{M}}_k(b^{*})\right)^{T}{b^{*}}\geq 0\\
&\implies \left(\nabla \hat{\mathcal{M}}_k(b^{*})\right)^{T}\left(\bar{b}-{b^{*}}\right)\geq 0.
\end{align*}

Thus, $\left(\nabla \hat{\mathcal{M}}(b^{*})\right)^{T}\left(\bar{b}-{b^{*}}\right)\geq 0$ for all $\bar{b}\in\mathcal{X}$. By the minimum principle, $b^{*}$ is an optimal solution to the convex problem \eqref{eqn:main_problem_one_node}. 
\end{proof}

\medskip
In the rest of this section, the sufficient condition (\ref{eqn:optimality-condition}) is used to justify the optimality of prespecified edge weight assignments for complete and tree graphs. 

\medskip
\begin{theorem}\label{edge_weight_complete_graph}
Suppose $G=(V,E)$ is a complete graph on $n$ nodes and fix $k\in{V}$. An optimal edge weight assignment that solves \eqref{eqn:main_problem_one_node} is a star centered at node $k$ with uniform edge weight distribution. 
\end{theorem}

\begin{proof}
Without loss of generality, set $k = 1$. We show that edge weight assignment $\bar{b}^{}$, where $\bar{b}^{}_l=\frac{1}{n-1}$ for all $j$ such that $l=(1,j)\in{E}$ and $\bar{b}^{}_l = 0$ otherwise, is optimal by showing that (\ref{eqn:optimality-condition}) holds for each edge $l$. Using gradient derivations presented in \cite{Ghosh} and the fact that $$\hat{\mathcal{M}}_1^{}\big(b\big)= {e_1}^{T}\left(\mathbb{L}(b)+\mathbf{1 1}^{T} / n\right)^{-1}{e_1} - \frac{1}{n}$$ 

\noindent
we have that $\frac{\partial \hat{\mathcal{M}}_1^{}\big(b^{}\big)}{\partial b^{}_{l}}$ where $l=(i,j)\in{E}$, equals

\[
    -{e_1}^{T}\left[\mathbb{L}(b)+\frac{\mathbf{1 1}^{T}}{n}\right]^{-1}\!\!\!\!\!\!{(e_i-e_j)}{(e_i-e_j)^{T}}\left[\mathbb{L}(b)+\frac{\mathbf{1 1}^{T}}{n}\right]^{-1}\!\!\!\!\!\!e_1, 
\]

\noindent for any $b\in{\mathcal{X}}$. Note that (\ref{eqn:optimality-condition}) where $k = 1$ holds true if and only if $-\frac{\partial\hat{\mathcal{M}}_1(b^{*})}{\partial b^{}_{l}} \leq \hat{\mathcal{M}_1}^{}(b^{*})$ for each $l=(i,j)$. Thus, (\ref{eqn:optimality-condition}) where $k = 1$ holds true if and only if
    \begin{align}\label{eqn:optimality-condition_eigenvalues}
    &\left(\sum_{\alpha={1}}^{n} v^{*}_{\alpha{1}}(v^{*}_{\alpha{i}}-v^{*}_{\alpha{j}})\lambda^{*}_{\alpha}\right)^{2}\leq \sum_{\alpha={1}}^{n} (v^{*}_{\alpha{1}})^{2}\lambda^{*}_{\alpha} - \frac{1}{n}
    \end{align}
 \noindent for each $l=(i,j)$, where $v^{*}_\alpha$ denotes the normalized eigenvector \footnote{That is, the two-norm of each eigenvector equals $1$.} associated with $\alpha^{th}$ eigenvalue of $\left(\mathbb{L}(b^{*})+\mathbf{1 1}^{T} / n\right)^{-1}$, $\lambda^{*}_{\alpha}$. By our edge weight assignment and properties of the spectrum of the Laplacian of a uniformly weighted star graph, the eigenvectors of $\left(\mathbb{L}(\bar{b})+\mathbf{1 1}^{T}/ n\right)^{-1}$ are 
 \begin{itemize}
     \item $\bar{v}_1 = \left[\frac{1}{\sqrt{n}},\dots,\frac{1}{\sqrt{n}}\right]^{T}$
     \medskip
     \item $\bar{v}_s = \frac{1}{\sqrt{2}}(e_s-e_{s+1}),$ for all $2\leq{s}\leq{n-1}$, and
     \medskip
     \item $\bar{v}_n = \left[\frac{n-1}{\sqrt{(n)(n-1)}},\frac{-1}{\sqrt{(n)(n-1)}},\dots, \frac{-1}{\sqrt{(n)(n-1)}}\right]^{T}$
 \end{itemize}
 \medskip
 \noindent with corresponding eigenvalues $\bar{\lambda}_1 = 1,\bar{\lambda}_s = n-1$ for $2\leq{s}\leq{n-1}$, and $\bar{\lambda}_n = \frac{n-1}{n}$. We now consider two cases.

 \medskip

\textbf{Case 1: }Suppose $i=1$ and $j\in\{2,\dots,{n}\}$ and consider $\sum_{\alpha={1}}^{n} \bar{v}_{\alpha{1}}(\bar{v}_{\alpha{1}}-\bar{v}_{\alpha{j}})\bar{\lambda}_{\alpha}$. Since $\bar{v}_{11},\bar{v}_{1j} = \frac{1}{\sqrt{n}}$, $\bar{v}_{{11}}(\bar{v}_{{11}}-\bar{v}_{{1j}})\bar{\lambda}_{1} = 0$. Moreover, since $\bar{v}_{\alpha{1}} = 0$ for all $2\leq{\alpha}\leq{n-1}$, 
     \begin{align*}
         \sum_{\alpha={1}}^{n} \bar{v}_{\alpha{1}}(\bar{v}_{\alpha{1}}-\bar{v}_{\alpha{j}})\bar{\lambda}_{\alpha}
         = \bar{v}_{{n1}}(\bar{v}_{{n1}}-\bar{v}_{{nj}})\bar{\lambda}_{n}
         =\frac{n-1}{n}.
     \end{align*}
     
     Thus, $\bigg(\sum_{\alpha={1}}^{n} \bar{v}_{\alpha{1}}(\bar{v}_{\alpha{1}}-\bar{v}_{\alpha{j}})\bar{\lambda}_{\alpha}\bigg)^{2} = \left(\frac{n-1}{n}\right)^{2}$, and 
     \begin{align*}
        \sum_{\alpha={1}}^{n} \bar{v}^{2}_{\alpha{1}}\bar{\lambda}_{\alpha} - \frac{1}{n}
        &= \left(\frac{n-1}{n}\right)^{2},
     \end{align*}
     
     \noindent implying that \eqref{eqn:optimality-condition_eigenvalues} is met. 

\medskip
 
\textbf{Case 2: }Suppose $i,j\in\{2,\dots,{n}\}$ and consider $\sum_{\alpha={1}}^{n} \bar{v}_{\alpha{1}}(\bar{v}_{\alpha{i}}-\bar{v}_{\alpha{j}})\bar{\lambda}_{\alpha}$. For this case, since $\bar{v}_{1{i}}=\bar{v}_{1{j}}$, $\bar{v}_{\alpha{1}}=0$ for ${2}\leq\alpha\leq{n-1}$, and $\bar{v}_{n{i}}=\bar{v}_{n{j}}$, $\sum_{\alpha={1}}^{n} \bar{v}_{\alpha{1}}(\bar{v}_{\alpha{i}}-\bar{v}_{\alpha{j}})\bar{\lambda}_{\alpha} = 0.$ Clearly,
$$0<\left(\frac{n-1}{n}\right)^{2}=\sum_{\alpha={1}}^{n} \bar{v}^{2}_{\alpha{1}}\bar{\lambda}_{\alpha} - \frac{1}{n},$$
\noindent
thus \eqref{eqn:optimality-condition_eigenvalues} is met. By \textbf{Case 1} and \textbf{Case 2}, (\ref{eqn:optimality-condition}) holds true for each edge $l$, and therefore, $\bar{b}=b^{*}$.
\end{proof}
\medskip

\begin{theorem}\label{edge_weight_tree}
Suppose $G=(V,E)$ is a tree with $|V| = n$ and $|E|= m$ and $k\in{V}.$ Let $\mathcal{A}$ be the set of all simple paths in $G$ and $\mathcal{A}^{k}$ be the set of all simple paths from node $k$ to all other nodes in $G$, then 
$$
\bar{b}^{}_l = \frac{(na^{k}_l-a_l)^{\frac{1}{2}}}{\sum_{s=1}^{m} (na^{k}_s-a_s)^\frac{1}{2}},
$$
\noindent for each edge $l\in\{1,\dots,m\}$ is an optimal edge weight assignment that solves \eqref{eqn:main_problem_one_node}, where $a_l$ ($a^{k}_l$) is the number of times edge $l$ appears in $\mathcal{A}$ ($\mathcal{A}^{k}$) . 
\end{theorem} 

\begin{proof}
We show that the edge weight assignment

$$
\bar{b}^{}_l = \frac{(na^{k}_l-a_l)^{\frac{1}{2}}}{\sum_{s=1}^{m} (na^{k}_s-a_s)^\frac{1}{2}},
$$

\noindent
for each edge $l\in\{1,\dots,m\}$ is an optimal solution to \eqref{eqn:main_problem_one_node}. Since $G=(V,E)$ is a tree, for all $(i,j)\in{E}$ and $b\in{\mathcal{X}}$, $\Omega_{ij}(b)$ is the sum of the reciprocal of edge weights that lie on the unique path from node $i$ to node $j$ \cite{Ghosh}. Thus, 

\begin{align*}
n^{2}\hat{\mathcal{M}}_k^{}\big(b\big) &= n^{}\sum_{j} \Omega_{j k}(b)-\sum_{i<j} \Omega_{i j}(b)\\
&= n\sum_{l=1}^{m} \frac{a^{k}_l}{b_l} - \sum_{l=1}^{m} \frac{a_l}{b_l} \\
&= \sum_{l=1}^{m} \frac{{n{a^{k}_l}}-a_l}{b_l}. 
\end{align*}

\noindent Note that  

\begin{align*}
    \frac{\partial{n^{2}}\hat{\mathcal{M}}_k(b^{})}{\partial b^{}_{l}} + {n^{2}}\hat{\mathcal{M}}_k^{}(b^{})
    = \frac{-(na_l^{k}-a_l)}{(b^{}_l)^{2}} + \sum_{s=1}^{m} \frac{{n{a^{k}_s}}-a_s}{b_s}.
\end{align*}

\noindent
It can be verified that by plugging in $\bar{b}$ as defined above we obtain,
$$\frac{-(na_l^{k}-a_l)}{(\bar{b}^{}_l)^{2}} = - \left(\sum_{s=1}^{m} \left({n{a^{k}_s}}-a_s\right)^{\frac{1}{2}}\right)^{2},\text{ and}$$
$$\sum_{s=1}^{m} \frac{{n{a^{k}_s}}-a_s}{\bar{b}_s} = \left(\sum_{s=1}^{m} \left({n{a^{k}_s}}-a_s\right)^{\frac{1}{2}}\right)^{2}.$$

\noindent By substitution, $\frac{\partial{n^{2}}\hat{\mathcal{M}}_k(\bar{b}^{})}{\partial \bar{b}^{}_{l}} + {{n^{2}}\hat{\mathcal{M}}_k}^{}(\bar{b}^{}) = 0$ implying that (\ref{eqn:optimality-condition}) holds true. Thus, $\bar{b}=b^{*}$.

\end{proof}
\color{black}


\section{High Voltage Electric Grid Application} \label{High Voltage Electric Grid Application}

An electrical grid is an interconnected network, consisting of transmission lines and buses, designed for the purpose of delivering power from producers to consumers. Power is delivered from generator buses to load buses via the transmission lines by way of alternating current, depicted by a sinusoidal curve. Electrical impedance is the measure of opposition that a transmission line presents to alternating current when a voltage is applied. Specifically, in the power grid, impedance of transmission lines are comprised of resistance and reactance. Buses where power is generated are called generator buses and have net positive power injections into the system whereas buses where power is consumed are called load buses and have net negative power injections into the system. 
 
Suppose we are considering a high voltage electric grid consisting of $n$ buses and $m$ transmission lines. We may model the grid as a connected network, $G = (V,E)$, with $n$ nodes and $m$ edges where each node $i\in{V}$ corresponds to a voltage phase angle $\theta_{i} \in[-\pi, \pi)$, associated with a bus $i$, and evolves according the coupled dynamics~\cite{Guo,Strogatz,Hu} described in \eqref{eqn:coupledoscillator1} where $\omega_{i}$ is the per unit power injected at node $i$, and $\sum_{j=1}^{N}b_{i j} \sin (\theta_{i}-\theta_{j})$ is the per unit power extracted at node $i$.\footnote{Note that there would be a damping term that we assume to be $1$ second for now.} We note that $b_{i j}$ is the per unit susceptance along the transmission line that connects bus $i$ to bus $j$. Concretely, $b_{i j}$ is calculated by taking the reciprocal of the per unit reactance along the transmission line that connects bus $i$ to bus $j$ in the physical system, and intuitively describes how conductive the transmission line is.

As previously stated, the ability for voltage phase-oscillators participating in the system to maintain frequency synchronization is imperative to the health and functionality of the electric grid, and this ability is potentially threatened by the integration of renewable energy into the high voltage electric grid.Indeed, variability in renewable output cause small disturbances to the power injected into the system at certain buses \cite{Tyloo,Impram}. To mitigate these small disturbances, we leverage the model as a design tool for ensuring the grid is robust to small perturbations that are inevitable with integration of renewable energy. We consider two scenarios whose solutions amount to ensuring that a fixed amount of susceptance\footnote{Normalized to one, for simplicity, and informed by a fixed amount of physical resources.} is optimally allocated to the edges of a power grid topology to minimize the vulnerability measure at nodes where renewable energy is introduced. As indicated in \cite{cortes2014virtual,rodriguez2013control}, electronics that control susceptance values along the transmission lines are currently in development, which would allow for this optimization strategy to be realized. Moreover, a motivating factor for this study is to explore the potential of new technology; We hope that the efficacy demonstrated by our optimization framework may inspire further development of these types of control electronics for the high voltage electric grid.

\subsection{Data set Description} 
We consider a $57$ bus case system that is a high voltage electric grid model for the NY region \cite{liu}. The dataset consists of $29$ generator buses, $28$ load buses, and $94$ transmission lines. To establish a natural frequency corresponding to each voltage phase-oscillator participating in the system, we attain the per unit power injected information associated with each node during a cold morning in December $2019$. For all of the problem scenarios considered, following the discussion in Section \ref{Ensuring Synchronicity}, we set $\gamma = \frac{\pi}{16}$ \cite{reliabilityNERC2019}.

\subsection{Scenario 1} 

\begin{quote}
Suppose a power grid engineer is tasked with converting the energy source at a generator bus to a renewable energy source which will likely result in small perturbations to the power injected at that bus. Is there a generator bus that would be the most robust to the introduction of renewable energy? Does the choice of bus change given the ability to distribute a fixed amount of susceptance to the edges of the electric grid network?
\end{quote}

        
Let $b_0$ be the original susceptance values obtained from data and suppose $V'$ is the set of nodes corresponding to the twenty-nine generator buses in the system. Without the ability to distribute a fixed amount of susceptance to the edges of the electric grid network, a generator bus that would be the most robust to the introduction of renewable energy is a generator bus corresponding to a node that solves $\argmin_{k\in{V'}}\hat{\mathcal{M}}_{k}(b_0)$. This is a direct application of the work done in \cite{Tyloo}. Suppose $b_{k}^{*}\in{\mathcal{X}}$ is a susceptance value assignment that minimizes the vulnerability measure at node $k$, i.e, $b_{k}^{*} = \argmin_{b\in\mathcal{X}}\hat{\mathcal{M}}_{k}(b)$. Given the ability to allocate susceptance to the edges of the high voltage electric grid, a generator bus that would be the most robust to the introduction of renewable energy corresponds to a node that solves $\argmin_{k\in{V'}}\hat{\mathcal{M}}_{k}(b_{k}^{*})$. 

We computed $\hat{\mathcal{M}}_{k}(b_0)$ and $\hat{\mathcal{M}}_{k}(b_{k}^{*})$ for all twenty-nine nodes corresponding to generator buses, and as expected, $\hat{\mathcal{M}}_{k}(b_0)$ is greater than $\hat{\mathcal{M}}_{k}(b_{k}^{*})$ for all $k$. Comparing $\hat{\mathcal{M}}_k(b_{k}^{*})$ for each node $k$ corresponding to a generator bus, the generator bus indexed as node $6$ exhibits the smallest vulnerability measure, $\hat{\mathcal{M}}_6(b_{6}^{*})$, and the generator bus indexed as node $4$ exhibits the largest vulnerability measure, $\hat{\mathcal{M}}_4(b_{4}^{*})$, after applying the optimization framework. 

Thus, given the ability to allocate susceptance to the edges of the high voltage electric grid, the generator bus that would be the most robust to the introduction of renewable energy corresponds to the node indexed as $6$. Moreover, without the ability to distribute a fixed amount of susceptance to the edges of the electric grid network, the generator bus that would be the most robust to the introduction of renewable energy is indexed as node $15$. We note that $\hat{\mathcal{M}}_{15}(b_0)$ is $87.3\%$ larger than $\hat{\mathcal{M}}_6(b_{6}^{*})$.

Given that renewable energy resources are often concentrated in areas according to natural resources, it may be unreasonable to assume that all twenty-nine generator buses should be considered as candidates for the introduction of renewable energy under this particular scenario. Our framework accommodates this constraint. In fact, one can choose to solve $b_{k}^{*}$ for each of the $k$ nodes corresponding the generator buses where appropriate natural resources are available, instead of all $k$ nodes corresponding to twenty-nine generator buses, and solve for $\argmin_{k\in{V'}}\hat{\mathcal{M}}_{k}(b_{k}^{*})$. 
\color{black}

\begin{figure}[H]
\centerline{\includegraphics[width=\columnwidth]{Scenario1.pdf}}
\caption{In the bar plot, $\hat{\mathcal{M}}_{k}(b_0)$ for each of the twenty-nine nodes corresponding to generator buses, $k$, is plotted in blue, and $\hat{\mathcal{M}}_{k}(b_{k}^{*})$ for each of the twenty-nine nodes corresponding to generator buses is plotted in orange. Below the bar plot, the $57$ bus-case system is plotted twice. In the left most graph, edge weights correspond to the edge weight assignment for node $4$ after the application of our optimization framework, and in the right most graph, edge weights correspond to the edge weight assignment for node $6$ after the application of our optimization framework. Notably, $\hat{\mathcal{M}}_{4}(b_{4}^{*})\approx 36.42$ and $\hat{\mathcal{M}}_{6}(b_{6}^{*}) \approx 12.68$.}
\label{fig3}
\end{figure}

\subsection{Scenario 2} 
\begin{quote}
Generating power using renewable energy resources rather than fossil fuels reduces greenhouse gas emissions, and thus, helps address climate change~\cite{Nunez}. Incorporating renewable energy at all the generator buses in the system will, however, likely result in small perturbations to the power injection at all these nodes. Can we distribute a fixed amount of susceptance to the edges of the electric grid network in such a way that allows for the voltage phase-oscillators' synchronized frequencies to be robust to noise at any of the generator buses?
\end{quote}

This problem amounts to solving \eqref{eqn:main_problem_specified} for $b^{*}$ where $V'$ is the set of $29$ nodes corresponding to generator buses in the complex network. Note that $\sum_{k\in{V'}}\hat{\mathcal{M}}_{k}(b_0)\approx 5663.14 $ and  $\sum_{k\in{V'}} \hat{\mathcal{M}}_{k}(b^{*})\approx 2628.66$, amounting to approximately a $53.6\%$ decrease in the sum of vulnerability measures at nodes corresponding to generator buses after applying the optimization framework. As illustrated by the plot in Figure~\ref{fig4}, the vulnerability measure at each bus $k\in{V'}$ decreases after the optimization framework except for three generator buses indexed as node $7,12,$ and $15$. Here, we are minimizing the worst case vulnerability measure for nodes in $V'$, so it is interesting to note that $\hat{\mathcal{M}}_k\big(b^{*}\big)$ becomes smaller for nearly all $k\in{V'}$. We would like to explore this aspect further as a potential future direction.

Once again, recall that renewable energy resources are concentrated in areas according to natural resources, and so it may be unreasonable to include all $29$ nodes corresponding to generator buses into the vertex subset $V'$. Instead, we may choose nodes corresponding to generator buses where natural resources are available to include in our vertex subset $V'$ to accomodate for such a natural resource constraint.

\begin{figure}[H]
\centerline{\includegraphics[width=\columnwidth]{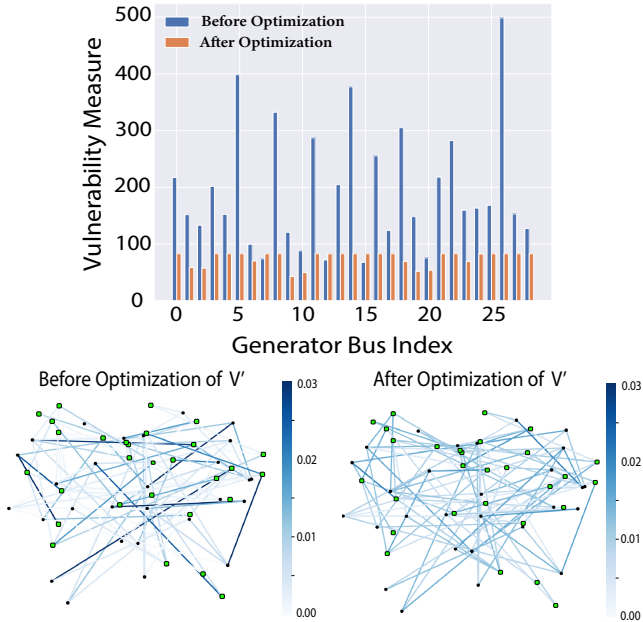}}
\caption{In the bar plot, , $\hat{\mathcal{M}}_{k}(b_0)$ for each of the twenty-nine nodes corresponding to generator buses, $k$, is plotted in blue, and $\hat{\mathcal{M}}_{k}(b^{*})$ for each of the twenty-nine nodes corresponding to generator buses is plotted in orange. Below the bar plot, once again, the $57$ bus-case system is plotted twice. In the left graph, edge weights correspond to original susceptance values, $b_0$, and in the right graph edge weights correspond to $b^{*}$.}
\label{fig4}
\end{figure}

\section{Conclusion}

In this work, we considered a small angle variation of the vulnerability measure derived in \cite{Tyloo} that quantifies how much a small perturbation to a phase-oscillator's natural frequency impacts the system's global synchronized frequencies. Given a fixed total amount of edge weight, we proposed a mathematical framework that assigns an optimal allocation of edge weights to minimize the vulnerability measure at node $k$, $\hat{\mathcal{M}}_{k}(b)$, or a function of vulnerability measures corresponding to a subset of nodes $V'$, $\mathcal{{F}}$, for which we expect small perturbations to occur. The model allows for flexibility in the choice of $\mathcal{{F}}$ contingent on the desired definition of robustness. In this work we specified $\mathcal{F}$ to produce edge weights that optimally minimize the worst case vulnerability measure of nodes in $V'$, $\max_{k\in{V'}}\hat{\mathcal{M}}_{k}(b)$. 

We proved that the vulnerability measure considered in this work is convex with respect to the edge weights of the network, implying that any edge weight assignment that results from the specified optimization problem is a global minimizer. Additionally, this work provided a tractable SDP reformulation of the problem and incorporated a constraint that ensures the existence of a synchronized stable solution with small angle differences. We shed light on the results of this optimization problem by considering the vulnerability of a single node from a graph theoretical and analytical lens. Finally, we applied the framework to high voltage electric grids, addressing two scenarios that highlight how the mathematical model may be leveraged to alleviate tensions between current green initiatives and the high voltage electric grids' capacity to accommodate such initiatives. 

There are many natural extensions to this work, both theoretical and applied in nature. One theoretical question to investigate is whether the vulnerability measure considered in this work is strictly convex with respect to the edge weights of a graph. If this property holds, this would imply that the solution obtained from the mathematical framework is a unique global minimizer. In Section~\ref{Analytical Solutions for Canonical Graphs}, we derived sufficient conditions for optimality when the vulnerability of one node is considered. It would be interesting to leverage these techniques to derive sufficient conditions for optimality when the vulnerability of a set of nodes is considered.

On the more applied side, recall that in establishing a natural frequency corresponding to each voltage phase-oscillator participating in the high voltage electric grid, we attained the per unit power injected information associated with a specific time. In reality, the high voltage electric grid is a highly dynamic system where the power injected at each bus varies in time. Thus, it would be informative to analyze how (if at all) the susceptance values assigned along the transmission lines vary in reference to time-series power injection data. If the susceptance values assigned along the transmission lines vary in reference to time-series power injection data, one could further quantify the variance and construct structures that minimize the variance of assigned susceptance values. Moreover, given that the high voltage electric grid is constantly growing in size, ensuring that the computational efficiency of the framework remains intact is important. One could potentially enhance the computational efficiency of this framework by taking advantage of the sparsity of $S_k$ discussed in Section~\ref{SDP Formulation}.

\appendices
\section{Proof of Proposition~\ref{prop:lower-bound-on-M}}\label{sec:proof-of-prop}
\begin{proof} 
Suppose $k\in{V}$ and $b\in{\mathcal{X}}$. In \cite{VanMieghem} and \cite{Gross}, the authors' show that, respectively, 

$$\hat{\mathcal{M}}_k(b) = \mathbb{L}^{\dagger}_{kk}\geq\frac{1}{\mathbb{L}^{}_{kk}} \left(1-\frac{1}{n}\right)^{2}\text{, and}$$

$$\lambda_2\geq\frac{1}{nD}$$

\noindent where $\mathbb{L}^{}_{kk}$ is the weighted degree of node $k$ and $D$ is the weighted diameter of the network being considered. Note that $\mathbb{L}^{}_{kk},D\leq{1}$ since $b^{T}\mathbf{1} = 1$, and so, 

$$\lambda_2{n}\geq{1}\geq{\mathbb{L}^{}_{kk}} \Leftrightarrow \frac{1}{\mathbb{L}^{}_{kk}} \geq \frac{1}{\lambda_{2}n}.$$

This means that 
$$\hat{\mathcal{M}}_k(b)\geq \frac{1}{\lambda_2} \left(1-\frac{1}{n}\right)^{2}$$
for all $k\in{V}$ when $b\in{\mathcal{X}}$, implying that for any ${V'}\subset{V}$,

$$\max_{k\in{V'}}\hat{\mathcal{M}}_k(b)\geq \frac{1}{\lambda_2} \left(1-\frac{1}{n}\right)^{2}.$$
\end{proof}
\color{black}

\section{Graph Theoretic Analysis Continued}\label{Graph Theoretic Interpretation Continued}
\begin{figure}[H]
\centerline{\includegraphics[width=\columnwidth]{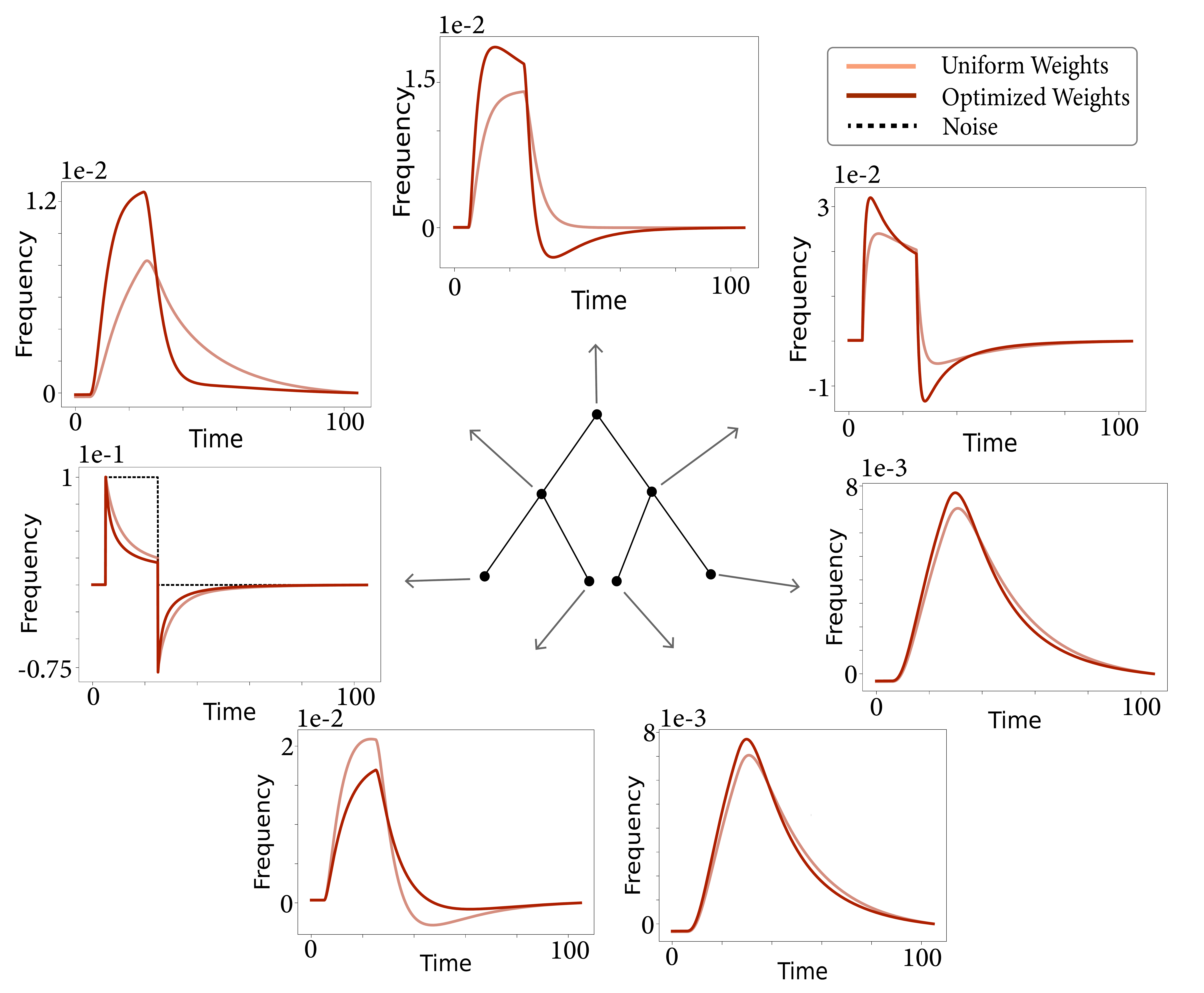}}
\caption{The node that is squared on $T_7$ in the center of this image is the node where we perturb the natural frequency for both the uniform edge weight case and optimized edge weight case. Each node on the graph has an associated arrow which points to a plot where the oscillators' frequencies over time for both cases, uniform and optimized edge weights are considered. For each of these plots, we consider time (seconds) on the $x$-axis and frequency (in a co-rotating frame) on the $y$-axis.}
\end{figure}

\section{Analytical Solutions for Canonical Graphs Continued}\label{Analytical Solutions for Canonical Continued}
\begin{lemma}\label{gradient_of_measure_property}
$\left(\nabla \hat{\mathcal{M}}_k^{}(b)\right)^{T} b=-\hat{\mathcal{M}}_k^{}(b)$ for $b\in{\mathcal{X}}$.
\end{lemma}

\begin{proof} For any $(i,j)\in{E}$ and any $c>0$ it is shown in \cite{Ghosh} that the effective resistance satisfies, $$\Omega_{i j}(cb) = \frac{\Omega_{i j}(b)}{c}.$$ 

\noindent
From this and the definition of $\hat{\mathcal{M}}_k^{}\big(b\big)$, it can easily be verified that 
\begin{equation}\label{measure_homogenous_degree_1}
    \hat{\mathcal{M}}_k^{}\big(cb\big)=\frac{1}{c} \hat{\mathcal{M}}_k^{}\big(b\big). 
\end{equation}

\noindent 
By differentiating both sides of~\eqref{measure_homogenous_degree_1} with respect to $c$ and then, setting $c = 1$, we obtain
$$\left(\nabla \hat{\mathcal{M}}_k^{}(b)\right)^{T} b=-\hat{\mathcal{M}}_k^{}(b).$$
\end{proof}

\color{black}


\section*{Acknowledgment} S. V. Nagpal and G. Nair would like to thank a few members from the Center for Applied Mathematics community at Cornell University for helpful conversations at varying stages in this work: Steve Strogatz, Maximilian Ruth, Shawn Ong, Misha Padidar, and Zachary Frangella. S. V. Nagpal and C. L. Anderson would like to acknowledge the Cornell Atkinson Center for Sustainability and the Cornell Energy Systems Institute funds for support. Finally, S. V. Nagpal would like to acknowledge the NSF Research Training Group Grant: Dynamics, Probability, and PDEs in Pure and Applied Mathematics, DMS-1645643 for partially funding this work.

\bibliographystyle{IEEEtran}
\bibliography{master}

\end{document}